\documentclass[a4paper,reqno,11pt]{amsart}
\usepackage{amsmath}
\usepackage{amssymb}
\usepackage{amsthm}
\usepackage{latexsym}
\usepackage{pstricks}
\usepackage{amsbsy}
\usepackage[all]{xypic}
\usepackage{amscd}
\usepackage{mathrsfs}
\usepackage{tipa}
\usepackage{txfonts}
\usepackage{amsfonts}
\usepackage{graphicx}
\usepackage{listings}
\usepackage{url}
\usepackage{bookmark}

\numberwithin{equation}{section}

\newtheorem{theorem}{Theorem}[section]
\newtheorem{lemma}[theorem]{Lemma}
\newtheorem{corollary}[theorem]{Corollary}
\newtheorem{proposition}[theorem]{Proposition}
\newtheorem{definition}[theorem]{Definition}
\newtheorem{remark}[theorem]{Remark}

\begin{document}

\providecommand{\ann}{\mathop{\rm ann}\nolimits}%
\providecommand{\gld}{\mathop{\rm gl. dim}\nolimits}%
\providecommand{\gord}{\mathop{\rm gor. dim}\nolimits}%
\providecommand{\pd}{\mathop{\rm pd}\nolimits}%
\providecommand{\rk}{\mathop{\rm rk}\nolimits}%
\providecommand{\Fac}{\mathop{\rm Fac}\nolimits}%
\providecommand{\ind}{\mathop{\rm ind}\nolimits}%
\providecommand{\Sub}{\mathop{\rm Sub}\nolimits}%
\providecommand{\coker}{\mathop{\rm coker}\nolimits}%
\providecommand{\cone}{\mathop{\rm cone}\nolimits}%
\def\s{\stackrel}
\def\A{\mathcal{A}}
\def\H{\mathcal{H}}
\def\C{\mathcal{C}}
\def\D{\mathcal{D}}
\def\DA{{D^b(A)}}
\def\dh{{D^b(H)}}
\def\DAA{{D^b(\A)}}
\def\K{{K^b(\proj A)}}
\def\H{\mathcal{H}}
\def\T{\mathcal{T}}
\def\P{\mathcal{P}}
\def\X{\mathcal{X}}
\def\Y{\mathcal{Y}}
\def\F{\mathcal{F}}
\def\R{\mathcal{R}}
\def\L{\mathcal{L}}
\def\U{\mathcal{U}}
\def\V{\mathcal{V}}
\def\W{\mathcal{W}}

\def\cc{\mathscr{C}}
\def\oo{\mathscr{O}}
\def\pp{\mathscr{P}}
\def\CP{\C(P^\bullet)}
\def\op{\text{op}}
\providecommand{\add}{\mathop{\rm add}\nolimits}%
\providecommand{\ann}{\mathop{\rm ann}\nolimits}%
\providecommand{\End}{\mathop{\rm End}\nolimits}%
\providecommand{\Ext}{\mathop{\rm Ext}\nolimits}%
\providecommand{\Hom}{\mathop{\rm Hom}\nolimits}%
\providecommand{\im}{\mathop{\rm im}\nolimits}%
\providecommand{\inj}{\mathop{\rm inj}\nolimits}%
\providecommand{\proj}{\mathop{\rm proj}\nolimits}%
\providecommand{\rad}{\mathop{\rm rad}\nolimits}%
\providecommand{\soc}{\mathop{\rm soc}\nolimits}%
\providecommand{\thick}{\mathop{\rm thick}\nolimits}%
\providecommand{\Tr}{\mathop{\rm Tr}\nolimits}%
\renewcommand{\dim}{\mathop{\rm dim}\nolimits}%
\renewcommand{\mod}{\mathop{\rm mod}\nolimits}%
\renewcommand{\ker}{\mathop{\rm ker}\nolimits}%
\renewcommand{\rad}{\mathop{\rm rad}\nolimits}%
\def \text{\mbox}
\def\ends{\end{enumerate}}
\newcommand{\id}{\operatorname{id}}
\renewcommand{\k}{\mathbf{k}}
\newcommand{\p}{\mathbf{P}}
\newcommand{\ii}{\mathbf{I}}
\newcommand{\q}{\mathbf{Q}}
\newcommand{\rr}{\mathbf{R}}
\newcommand{\x}{\mathbf{X}}
\newcommand{\y}{\mathbf{Y}}
\renewcommand{\c}{\mathbf{C}}
\renewcommand{\d}{\mathbf{D}}
\newcommand{\e}{\mathbf{E}}
\renewcommand{\t}{\mathbf{T}}


\title{Silted algebras}

\author[Buan]{Aslak Bakke Buan}
\address{
Department of Mathematical Sciences
Norwegian University of Science and Technology
7491 Trondheim
Norway
}
\email{aslakb@math.ntnu.no}

\author[Zhou]{Yu Zhou}
\address{
Department of Mathematical Sciences
Norwegian University of Science and Technology
7491 Trondheim
Norway
}
\email{yu.zhou@math.ntnu.no}

\begin{abstract}
We study endomorphism algebras of 2-term silting complexes in
derived categories of hereditary finite dimensional algebras,
or more generally of $\Ext$-finite hereditary abelian categories.
Module categories of such endomorphism algebras are known to occur as hearts of certain bounded $t$-structures in
such derived categories.
We show that the algebras occurring are exactly
the algebras of small homological dimension, which are algebras characterized by
the property that each indecomposable module either has
injective dimension at most one, or it has projective dimension at most one.
\end{abstract}

\thanks{
This work was supported by FRINAT grant number 231000, from the
Norwegian Research Council. Support by the Institut Mittag-Leffler (Djursholm, Sweden) is gratefully
acknowledged.
}

\keywords{Silting theory; Torsion pairs; Shod algebras; Derived categories.}

\maketitle

\section*{Introduction}

Happel and Ringel \cite{hr} introduced {\em tilted algebras}
in the early eighties. These are the finite dimensional algebras which occur as endomorphism algebras of tilting
modules over hereditary finite dimensional algebras.

The notion was generalized by Happel, Reiten and
Smal{\o} \cite{hrs}, who introduced {\em quasi-tilted algebras},
which are the algebras occurring as endomorphism algebras of tilting objects in hereditary abelian
categories with finiteness conditions.
Quasi-tilted algebras were shown to
have a natural homological characterization.
Consider the following property on the category
$\mod A$,  of finite dimensional right $A$-modules,
for a finite dimensional algebra $A$:
for each indecomposable module $X$, we
have either $\pd X \leq 1$ or $\id X \leq 1$ (that is: the projective or the injective dimension is at most one).
Let us call this the {\em shod}-property (=small homological
dimension).

In \cite{hrs}, they showed that the quasi-tilted algebras
are exactly the algebras of global dimension at most two
with the shod-property. They also showed that algebras
with the shod-property have global dimension at most three.

Then, Coelho and Lanzilotta \cite{cl} defined {\em shod algebras}, as the algebras with the shod-property.
Shod algebras of global dimension three, they called
{\em strictly shod}.
Later Reiten and Skowro\'{n}ski \cite{rs}, gave a characterization
of the strictly shod algebras, in terms of a property of
the AR-quiver $Q_A$ of $A$: namely the existence of what they called a faithful double section in $Q_A$.

Our aim is to show that shod algebras admit
a very natural characterization, using the notion
of silting complexes, as introduced by Keller and
Vossieck \cite{kv}. Let $\p$ be a complex in
the bounded homotopy category of finitely generated projective $A$-modules
$K^b(\proj A)$. Then $\p$ is called {\em silting}
if $\Hom_{\K}(\p, \p[i])= 0$ for $i>0$, and if $\p$ generates
$\K$ as a triangulated category.
Furthermore, we say that $\p$ is {\em 2-term} if $\p$ only has non-zero terms in degree $0$ and $-1$.
In Section \ref{sec:abel}, we generalize the notion to bounded
derived categories of abelian categories.
Note that $\K$ can be considered as a full subcategory of
the bounded derived category $\DA$, and as equal to
$\DA$ if $A$ has finite global dimension.

\begin{definition}
Let $B$ be a finite dimensional algebra over a field $\k$. We call $B$ {\em silted} if there is a finite dimensional hereditary algebra $H$ and a 2-term silting complex $\p\in K^b(\proj H)$ such that $B \cong \End_{D^b(H)}(\p)$.
Furthermore, the algebra $B$ is called quasi-silted if
$B \cong \End_{D^b(\H)}(\p)$ for a 2-term silting complex
$\p$ in the derived category of an $\Ext$-finite hereditary abelian category $\H$.
\end{definition}

In \cite{bz}, we studied torsion theories induced
by 2-term silting complexes, and generalized classical
results of Brenner-Butler \cite{bb} and Happel-Ringel \cite{hr}.
Here we apply these results to prove the following main result.

\begin{theorem}\label{main}
Let $A$ be a connected finite dimensional algebra over an
algebraically closed field $\k$. Then
\begin{itemize}
\item[(a)] $A$ is a strictly shod or a tilted algebra
if and only if it is a silted algebra.
\item[(b)] $A$ is a shod algebra if and only it is a quasi-silted algebra.
\end{itemize}
\end{theorem}

The paper is organized as follows.
We first recall some notation and facts concerning 2-term silting complexes and
induced torsion pairs.
Then in Section \ref{sec:silted-shod},
we prove part (a) of our main theorem.
In Section \ref{sec:double}
we provide a link to the classification of shod algebras by
Reiten and Skowro\'{n}ski.
Then we define 2-term silting complexes in bounded derived categories of $\Ext$-finite
abelian categories,
in Section \ref{sec:abel}, where
also part (b) of our main theorem is proved.
We conclude
with providing a small example to illustrate our main theorem. See \cite{ass} for the definition of
torsion pairs and other undefined notions for module categories. See \cite{h, hrs} for the definition
of t-structures and other undefined notions for
derived categories.
 
We would like to thank Steffen Oppermann and Dong Yang for
discussions related to this paper.

\section{Background and notation}\label{sec:silting}

In this section we fix notation and recall facts
concerning silting theory for 2-term
silting complexes. We refer to \cite{bz} for details.

In this paper, all modules are right modules.
Let $A$ be a finite dimensional algebra over a field $\k$. We denote by $\mod A$ the category of all finitely generated $A$-modules. A composition $fg$ of morphisms $f$ and $g$ means first $g$ and then $f$. But a composition $ab$ of arrows $a$ and $b$ means first $a$ then $b$. Under this setting, we have an equivalence from the category of all finite dimensional representations of a quiver $Q$ bounded by an admissible ideal $J$ to the category $\mod \k Q/J$, and a canonical isomorphism $A\cong \End_AA.$ For an $A$-module $M$, we let $\add M$ denote the full subcategory
of all direct summands in direct sums of copies of $M$, we let $\Fac M$ denote the full subcategory
of all modules which are factors of modules in $\add M$,
and we let $\Sub M$ denote the full subcategory
of all modules which are submodules of modules in $\add M$. Let $\rad M$ denote the
radical of a module $M$, and let $\soc M$ denote the socle of $M$.
We let $D= \Hom_{\k}(-,\k)$ denote the ordinary vector-space duality,
we let $\nu$ denote the Nakayama functor $\nu = D\Hom_A(-,A)$,
and we let $\tau_A$ denote the Auslander-Reiten translation in
$\mod A$. Note that the Nakayama functor induces an equivalence
$K^b(\proj A) \to K^b(\inj A)$, where $\proj A$ and $\inj A$ denote the
full subcategories of projectives and injectives in $\mod A$, respectively. Furthermore, we
have an isomorphism $$\Hom_{\DA}(\x,\nu \y) \cong D\Hom_{\DA}(\y,\x)$$
for $\x,\y$ in $\K$. 
Let $\U, \V$ be full subcategories of a triangulated
(or abelian) category $\W$. We let $\U \ast \V$ denote the full subcategory with objects 
 occurring as a middle terms of triangles (or short exact sequences) with left end terms in
$\U$ and right end terms in $\V$.

Consider a pair of indecomposable $A$-modules $X,Y$. If there exists a sequence of non-zero morphisms $X=X_0\xrightarrow{f_0}X_1\xrightarrow{f_1}X_2\xrightarrow{f_2}\cdots\rightarrow X_n\xrightarrow{f_n}  X_{n+1}=Y$ with $X_i$ indecomposable for $i=0,1,\cdots,{n+1}$, then we call $X$ a {\em predecessor} of $Y$, and $Y$ is called a {\em successor} of $X$.

Let $\p$ be a 2-term silting complex in $K^b(\proj A)$ for a finite dimensional $\k$-algebra $A$ and let $B = \End_{\DA}(\p)$.
We always assume $\p$ is basic, and that $\k$ is algebraically closed. Then
$B$ is a path algebra modulo an admissible ideal. Consider the subcategories
$$\T(\p) = \{X \in \mod A \mid \Hom_{\DA}(\p,X[1]) = 0 \}$$
and $$\F(\p) = \{X \in \mod A \mid \Hom_{\DA}(\p,X) = 0 \}.$$
Then $(\T(\p), \F(\p))$ is a torsion pair in $\mod A$.
Furthermore, the functors $\Hom_{\DA}(\p,-)$ restricted
to $\T(\p)$ and $\Hom_{\DA}(\p, -[1])$ restricted to $\F(\p)$ are both fully faithful and there is a 2-term silting complex $\q$ in
$D^b(\mod B)$, such that $\X(\p) \colon= \Hom_{\DA}(\p,\F(\p)[1])  =\T(\q)$ and $\Y(\p) \colon= \Hom_{\DA}(\p,\T(\p))
=\F(\q)$. We will refer to this fact, which is the main
result of \cite{bz}, as {\em the silting theorem}.

Note that $\Hom_{\DA}(\p, \nu \p)\cong D\Hom_{\DA}(\p, \p)$ is an injective cogenerator for $\mod B$.

The following facts from \cite{hkm} and \cite{bz} concerning the torsion pair $(\T(\p), \F(\p))$ are useful.

\begin{proposition}\label{prop:tf}
With notation as above, we have:
\begin{itemize}
  \item[(a)] $(\T(\p), \F(\p)) = (\Fac H^0(\p), \Sub H^{-1} (\nu\p))$;
\item[(b)] the modules in $\add H^0(\p)$ are the $\Ext$-projectives in $\T(\p)$;
\item[(c)] for each $X$ in $\T(\p)$, there is an exact sequence
$$0 \to L \to T_0 \to X \to 0$$
with $T_0$ in $\add H^0(\p)$ and $L$ in $\T(\p)$;
\item[(d)] the modules in $\add H^{-1}(\nu\p)$ are the $\Ext$-injectives in $\F(\p)$;
\item[(e)] for each $Y$ in $\F(\p)$, there is an exact sequence
$$0 \to Y \to F_0 \to L \to 0$$
with $F_0$ in $\add H^{-1}(\nu\p)$ and $L$ in $\F(\p)$.
\end{itemize}
\end{proposition}

We shall also need the following facts concerning
$B= \End_{\DA}(\p)$ and the torsion pair
$(\X(\p),\Y(\p))$ in $\mod B$, in the case where $A$ is hereditary.

\begin{proposition}\label{prop:basic}
Let $H$ be a hereditary algebra, let $\p$ be a 2-term silting complex in $K^b(\proj H)$ and let $B=\End_{D^b(H)}(\p)$. Then
the following hold.
\begin{itemize}
  \item[(a)] The torsion pair $(\X(\p),\Y(\p))$ is split.
  \item[(b)] $\X(\p)$ is closed under successors and $\Y(\p)$ is closed under predecessors.
  \item[(c)] For any $X\in\X(\p)$ and $Y\in\Y(\p)$, we have $\pd Y_B\leq 1$ and $\id X_B\leq1$.
  \item[(d)] The global dimension of $B$ is at most 3.
  \item[(e)] Any almost split sequence in $\mod B$ lies entirely in either $\X(\p)$ or $\Y(\p)$, or else it is a connecting sequence
  \[0\rightarrow F(\nu P)\rightarrow F((\rad P)[1])\oplus F(\nu P/\soc(\nu P))\rightarrow \\ F(P[1])\rightarrow0,\]
    for $P$ indecomposable projective with $P\notin\add(\p\oplus\p[1])$, where $F=\Hom_{D^b(H)}(\p,-)$.
\end{itemize}
\end{proposition}

\begin{proof}
(a) This follows from \cite[Lemma~5.5]{bz}.

\noindent (b) This follows from (a), using the fact that there
are no non-zero maps from objects in $\X(\p)$
to objects in $\Y(\p)$.

\noindent (c) Let $Y$ be an object in $\Y(\p)$. By definition, there is an $H$-module $M\in\T(\p)$ such that $\Hom_{D^b(H)}(\p,M)=Y$. Then by Proposition~\ref{prop:tf}(c), there is an exact sequence
\[0\rightarrow L\rightarrow T_0\rightarrow M\rightarrow0\]
with $T_0\in\add H^0(\p)$ and $L\in\T(\p)$. Applying $\Hom_{D^b(H)}(\p,-)$, we have an exact sequence
\[0\rightarrow\Hom_{D^b(H)}(\p,L)\rightarrow\Hom_{D^b(H)}(\p,T_0)\rightarrow\Hom_{D^b(H)}(\p,M)\rightarrow0.\]
On the other hand, applying $\Hom_H(-,N)$ for any $N\in\T(\p)$ we have an exact sequence
\[\Ext_H^1(T_0,N)\rightarrow\Ext_H^1(L,N)\rightarrow\Ext_H^2(M,N)\]
where the first term is zero, using Proposition 
\ref{prop:tf}(b) and that $T_0$ is in $\add H^0(\p)$, and where the last term is zero since $H$ is hereditary.
So $\Ext_H^1(L,N)=0$ for any $N\in\T(\p)$,
which implies $L\in\add H^0(\p)$ by Proposition~\ref{prop:tf}(b).
Since $\pd H^0(\p)_H\leq 1$,
we have that $H^0(\p)$ is isomorphic to a direct summand of $\p$.
So both $\Hom_{D^b(H)}(\p,T_0)$ and $\Hom_{D^b(H)}(\p,L)$ \sloppy are projective, hence,
we have $\pd Y_B=\pd\Hom_{D^b(H)}(\p,M)_B\leq 1$. The proof for $X\in\X(\p)$ is similar by using Proposition~\ref{prop:tf}(d,e).

(d) This follows from (c), using \cite[Proposition~2.1.1]{hrs}.

(e) This follows from (a) and \cite[Proposition~5.6]{bz}.
\end{proof}

We also need the following observation which follows from \cite[Lemma 2.3]{air}.

\begin{lemma}\label{lem:summand}
For a projective $A$-module $P$, we have that
$P[1]$ is a direct summand in
$\p$ if and only if $\Hom_A(P, H^0(\p)) = 0$.
\end{lemma}

We recall the notion of (strictly) shod algebras.

\begin{definition}[\cite{cl}]
An algebra $A$ is called shod if for each indecomposable $A$-module $X$, either $\pd X_A\leq1$ or $\id X_A\leq 1$. Also, $A$ is called strictly shod if $A$ is shod and $\gld A=3$.
\end{definition}

Let $\L$ be the set of the indecomposable $A$-modules $Y$ such that for all
predecessors $X$ of $Y$, we have that $\pd X_A\leq1$ and $\R$ be the set of the indecomposable $A$-modules $X$ such that for all
successors $Y$ of $X$, we have that $\id Y_A\leq1$. We recall some equivalent characterizations of shod algebras.

\begin{proposition}[\cite{cl}]\label{prop:CL}
The following are equivalent for an algebra $A$:
\begin{itemize}
  \item[(a)] $A$ is a shod algebra;
  \item[(b)] $(\add\R,\add(\L\setminus\R))$ is a split torsion pair in $\mod A$;
  \item[(c)] $(\add(\R\setminus\L),\add\L)$ is a split torsion pair in $\mod A$;
  \item[(d)] there exists a split torsion pair $(\T,\F)$ in $\mod A$ such that $\pd Y_A\leq1$ for each $Y\in\F$ and $\id X_A\leq1$ for each $X\in\T$.
\end{itemize}
\end{proposition}

\section{Silted algebras and shod algebras}\label{sec:silted-shod}

Note that by Proposition~\ref{prop:basic}(a,c),
any silted algebra is shod.
In this section we prove that a finite dimensional algebra
over an algebraically closed field is silted if and only
if it is strictly shod or tilted.
The following will be the key fact for our proof.

\begin{proposition}\label{prop:key}
If there is a 2-term silting complex $\p\in K^b(\proj A)$
such that $(\T(\p),\F(\p))$ satisfies condition (d) in Proposition~\ref{prop:CL}, then $A$ is a silted algebra.
\end{proposition}

The proof of this will follow after a series of lemmas.
For this we now fix an algebra $A$ and assume the existence of a
2-term silting complex $\p\ =\left(P^{-1}\xrightarrow{p}P^0\right) \in K^b(\proj A)$, such that
$(\T(\p),\F(\p))$
satisfies condition (d) in Proposition~\ref{prop:CL}. Then, in particular, $A$ is a shod algebra. Let $B=\End_\DA(\p) \cong \k Q_B/J_B$, where $J_B$ is an admissible ideal.
We aim to prove that $Q_B$, the Gabriel quiver of $B$,
is acyclic. Then we show that the corresponding hereditary
algebra $H =\k Q_B$ admits a 2-term silting complex, such that $A$
is isomorphic to its endomorphism ring. This is our strategy for proving
Proposition \ref{prop:key}.

In order to describe $Q_B$ and $J_B$ we need first to consider two
factor algebras of $B$, namely $\End_A(H^0(\p))$ and $\End_A(H^{-1}(\nu \p))$.

Let $\p=\p_L \oplus \p_M \oplus \p_R$, where $\p_L$ is the direct sum of the indecomposable direct summands of $\p$ with zero $0$th term and $\p_R$ is the direct sum of the indecomposable direct summands of $\p$ with zero $-1$th term. We shall need the following lemma.

\begin{lemma}\label{lem:pd2}
$H^{-1}(\p)$ is projective and $H^{-1}(\p)[1]\in\add\p_L$. Dually, $H^0(\nu \p)$ is injective and $H^0(\nu \p)\in\add\nu\p_R$.
\end{lemma}
\begin{proof}
We prove the first part. The proof of the second part is similar. Since $\p$ is given by the map $P^{-1}\xrightarrow{p} P^0$,
we have an exact sequence \[0\to H^{-1}(\p)\to P^{-1}\xrightarrow{p} P^0\xrightarrow{c_p} H^{0}(\p)\to 0.\]
By definition, the map $c_p$ is a projective cover of $H^0(\p)$.
Let $K=\ker c_p$.
By Proposition \ref{prop:tf}(b),
$H^0(\p)$ is Ext-projective in $\T(\p)$.
Therefore any direct summand of $K$ is not in $\T(\p)$,
and since $(\T(\p),\F(\p))$ is split, by the assumption on $\p$, we have that $K$ belongs to $\F(\p)$.
So $\pd K_A \leq 1$ by the assumption, and hence
$H^{-1}(\p)$ is projective.

We now prove that $H^{-1}(\p)[1]$ is in $\add \p_L$.
By Lemma \ref{lem:summand}, it is \sloppy sufficient to prove that $\Hom_A(H^{-1}(\p), H^0(\p))= 0$.
Applying $\Hom_\DA(-,H^0(\p))$ to the triangle
\[H^{-1}(\p)[1]\to\p\to H^0(\p)\to H^{-1}(\p)[2], \]
we have an exact sequence
\[\Hom_\DA(\p,H^0(\p)[1])\to \Hom_\DA(H^{-1}(\p)[1],H^0(\p)[1])\to \Hom_\DA(H^{0}(\p),H^0(\p)[2]),\]
where the first term is zero by $H^0(\p)\in\T(\p)$ and the third term is zero by $\id H^0(\p)_A\leq 1$. Then we have $\Hom_A(H^{-1}(\p), H^0(\p))= 0$.
\end{proof}

By this we get the following
relations between the algebras $\End_A(H^0(\p))$, $\End_A(H^{-1}(\nu \p))$ and $\End_\DA(\p)$.

\begin{lemma}\label{lem:surj}
The functor $H^0(-)$ gives a surjective homomorphism of algebras
\[H^0(-)\colon \End_\DA(\p)\rightarrow\End_A(H^0(\p))\]
whose kernel is the space consisting of morphisms which factor through $\add \p_L$. The functor $H^{-1}(\nu-)$ gives a surjective homomorphism of algebras
\[H^{-1}(\nu-)\colon \End_\DA(\p)\rightarrow\End_A(H^{-1}(\nu \p))\]
whose kernel is the space consisting of morphisms which factor through $\add \p_R$.
\end{lemma}

\begin{proof}
Since $H^0(-)$ is a $\k$-linear functor, it gives a homomorphism of the algebras. Using that $\p$ is a projective presentation of $H^0(\p)$, we obtain that this homomorphism is surjective. By $H^0(\p_L)=0$, we have that $H^0(f)=0$ for any morphism $f$ which factors through $\p_L$. Now we assume that $H^0(f)=0$ for a chain map $f=(f^{-1}\ f^0)\in\End_\DA(\p)$. Considering the following diagram:
\[\xymatrix@R=1pc{
P^{-1}\ar[rr]^{p}\ar[dd]^{f^{-1}}&&P^0\ar[dd]^{f^0}\ar[rr]&&H^0(\p)\ar[dd]^{H^0(f)}\ar[rr]&&0\\
\\
P^{-1}\ar[rr]^{p}&&P^0\ar[rr]&&H^0(\p)\ar[rr]&&0
}\]
we have that $f^0$ factors through $p$. So we may assume that $f^0=0$ up to homotopy. In this case, $f^{-1}$ factors through $H^{-1}(\p)$. Due to Lemma~\ref{lem:pd2}, we have $H^{-1}(\p)[1]\in\add \p_L$. Hence $f$ factors through $\add \p_L$. Thus, the proof of the first statement
is complete. The second statement is dual to the first one.
\end{proof}

The following will be crucial for proving that $Q_B$
is acyclic.

\begin{lemma}\label{lem:here}
Both $\End_A(H^0(\p))$ and $\End_A(H^{-1}(\nu \p))$ are hereditary
algebras.
\end{lemma}

\begin{proof}
We prove the statement for $B_0 =\End_A(H^0(\p))$.
The proof for $\End_A(H^{-1}(\nu \p))$ is similar.
To complete the proof, it is sufficient to prove that $\pd S_{B_0}\leq 1$ for any simple $B_0$-module $S$.
Since, by Lemma \ref{lem:surj},
the algebra $B_0$ is a factor algebra of $B$, we have
that $\mod B_0$ is a full subcategory of $\mod B$ closed under both submodules and
factor modules. Therefore, the torsion pair $(\X(\p),\Y(\p))$
in $\mod B$, gives rise to a torsion pair
$(\X(\p)\cap\mod B_0,\Y(\p)\cap\mod B_0)$ in $\mod B_0$. It is straightforward to verify that we have $\Hom_\DA(\p,H^0(\p))\cong\Hom_A(H^0(\p),H^0(\p))$, so $\Hom_\DA(\p,H^0(\p))$ is a projective generator of $\mod B_0$.

Let $S$ be a simple $B_0$-module and $\Hom_\DA(\p,T_0)\s{p_S}\rightarrow S\rightarrow 0$ be a projective cover of $S$, where $T_0\in\add H^0(\p)$. We have that $\Hom_\DA(\p, T_0)$ is in $\Y(\p)$, and since
$\Y(\p)$ is closed under submodules, also the kernel
of $p_S$ is in $\Y(\p)$. Hence there is an $L$ in $\T(\p)$,
such that there is an exact sequence
\begin{equation}\label{eq:ses2}
0 \to \Hom_\DA(\p,L) \to \Hom_\DA(\p,T_0) \to S \to 0.
\end{equation}
We claim that $\Ext^1_A(L,\T(\p))=0$.
Indeed, since $S$ is simple,
either $S\in\X(\p)\cap\mod B_0$ or $S\in\Y(\p)\cap\mod B_0$.
If $S\in\Y(\p)\cap\mod B_0$, then, by definition,
there is an $X\in\T(\p)$ such that $S=\Hom_\DA(\p, X)$.
By the silting theorem, the exact sequence \eqref{eq:ses2} gives rise to an exact sequence
\begin{equation}\label{eq:ses3}
0\rightarrow L\rightarrow T_0\rightarrow X\rightarrow 0
\end{equation}
in $\mod A$. By the assumption on $\T(\p)$, we have that \sloppy
$\Ext_A^2(-,\T(\p))=0$. In particular $\Ext_A^2(X,\T(\p))=0$, and hence
it follows from the sequence (\ref{eq:ses3}) that
$\Ext_A^1(L,\T(\p))=0$.

Now consider the case with $S\in\X(\p)\cap\mod B_0$.
Then, by definition, there is a $Y\in\F(\p)$ such that $S=\Hom_\DA(\p,Y[1])$.
So we have a triangle (see \cite[Theorem 2.3]{bz}) in $\DA$:
\[Y\rightarrow L\rightarrow T_0\rightarrow Y[1].\]
Since $(\T(\p),\F(\p))$ is split, it follows
that $\Ext^1(Y,\T(\p))=0$. Therefore $\Ext^1(L,\T(\p))=0$. Thus, we have finished the proof of the claim that $\Ext^1_A(L,\T(\p))=0$. This implies that $L\in\add H^0(\p)$ by Proposition~1.1(b).
Hence the exact sequence \eqref{eq:ses2} is a projective resolution of $S$,
which shows that $\pd S_{B_0}\leq 1.$
\end{proof}



\providecommand{\Irr}{\mathop{\rm Irr}\nolimits}%
\def\B{\mathbf{B}}%

We first give a preliminary description of
$Q_B$ and $J_B$ for $B = \End_\DA(\p)$. This will be improved in Lemma \ref{lem:ac}.

Let $V_?$ be the set of the vertices of $Q_B$ corresponding to the direct summands of $\p_?$ and let $e_?$ be the sum of primitive orthogonal idempotents $e_v$ with $v\in V_?$, where $?=L$, $M$ or $R$. Let $Q_B^{LM}$ (resp. $Q_B^{MR}$) be the full subquiver of $Q_B$ consisting of the vertices in $V_L\cup V_M$ (resp. $V_M\cup V_R$) and the arrows between them. For any algebra $\Lambda$, we use $Q_{\Lambda}$ to denote its (Gabriel) quiver.

\begin{lemma}\label{lem:qb-ver1}
With the above notation, the following hold. 
\begin{itemize}
\item[(a)] Each path in $Q_B$ from $V_L$ to $V_R$, or from $V_R$ to $V_L$, is in $J_B$. In particular, there are no arrows from $V_L$ to $V_R$ and no arrows from $V_R$ to $V_L$.
\item[(b)] The surjective algebra morphisms in Lemma~\ref{lem:surj} induce isomorphisms
of quivers $Q_B^{MR}\cong Q_{\End_A(H^0(\p))}$ and $Q_B^{LM}\cong Q_{\End_A(H^{-1}(\nu \p))}$.
\item[(c)] The algebra $B$ is monomial,
and the ideal $J_B$ is generated by the paths from $V_L$ to $V_R$ and the paths from $V_R$ to $V_L$.
\end{itemize}
\end{lemma}

\begin{proof}
The first assertion in (a) follows from $\Hom_\DA(\p_R,\p_L)=0$ and $\Hom_\DA(\p_L,\p_R)=0$. Then the second assertion follows since $J_B$ is admissible.

The surjective algebra morphisms in Lemma~\ref{lem:surj} induce \sloppy algebra-isomorphisms
$\End_A(H^0(\p))\cong B/Be_LB$ and $\End_A(H^{-1}(\nu\p))\cong B/Be_RB$. The (Gabriel) quiver of $B/Be_?B$ is obtained from $Q_B$ by removing the vertices in $V_?$ and the arrows adjacent to these vertices, where $?=L, R$. Then (b) follows.

Moreover, it follows from the above algebra-isomorphisms that $J_B/J_Be_LJ_B=0=J_B/J_Be_RJ_B$ since, by Lemma~\ref{lem:here}, $\End_A(H^0(\p))$ and $\End_A(H^{-1}(\nu\p))$ are hereditary. So for any minimal relation $\sum_i\lambda_if_i$ with $\lambda_i\in\k$ nonzero and $f_i$ a path in $Q_B$, each $f_i$ factors through $V_L$ and $V_R$. Hence, by (a), the assertions in (c) follow.
\end{proof}

Before we can show that $Q_B$ is acyclic we need some more
properties of the torsion pair $(\X(\p), \Y(\p))$.

\begin{lemma}\label{lem:left-right}
With the above notation, the following hold. 
\begin{itemize}
\item[(a)] For any $X$ in $\mod A$, we have
that the $B$-module $\Hom_{\DA}(\p,X[1])$ is in $\X(\p)$, and that
$\Hom_{\DA}(\p,X)$ is in $\Y(\p)$.
\item[(b)] We have that the projective $B$-module $\Hom_{\DA}(\p,\p_L)$
is in $\X(\p)$ and that the injective $B$-module
$\Hom_{\DA}(\p, \nu \p_R)$
is in $\Y(\p)$.
\item[(c)] Let $C_M$ be the projective cover
of a $B$-module $M$. If $C_M$ is in
$\add \Hom_{\DA}(\p,\p_L)$, then $M$ is in $\X(\p)$.
\item[(d)] Let $E_M$ be the injective envelope
of a $B$-module $M$. If $E_M$ is in
$\add \Hom_{\DA}(\p, \nu \p_R)$, then $M$ is in $\Y(\p)$.
\end{itemize}
\end{lemma}

\begin{proof}
Part (a) is contained in \cite[Lemma~3.6]{bz}.
Part (b) follows directly from (a).
Then parts (c) and (d) follow from the
facts that $\X(\p)$ is closed under factor modules,
and $\Y(\p)$ is closed under submodules.
\end{proof}

\begin{lemma}\label{lem:splits}
We have $\Ext^2_B(\X(\p), \Y(\p)) =0$.
\end{lemma}

\begin{proof}
Since $(\T(\p), \F(\p))$ is split by assumption,
we have that $\p$ is a tilting complex by \cite[Proposition 5.7]{bz}. Then we have \[\Ext_B^2(\X(\p),\Y(\p))\cong \Hom_{D^b(A)}(\F(\p)[1],\T(\p)[2])\cong\Ext^1_A(\F(\p),\T(\p))=0.\]
\end{proof}

\begin{lemma}\label{lem:nopath}
There are no paths from $V_L$ to $V_R$ in $Q_B$.
\end{lemma}

\begin{proof}
Assume there is a path $p$ from $v_1$ in $V_L$ to $v_2$ in $V_R$, and assume
it contains no proper subpath from $V_L$ to $V_R$. Then, by
Lemma \ref{lem:qb-ver1} (c), it follows that $p$ is a minimal relation. Let $S_i$ for $i=1,2$
be the simple $B$-module corresponding to vertex $v_i$.
Then $\Ext^2_B(S_1, S_2) \neq 0$, by \cite[Proposition~3.4]{birs} (note that we use right modules but they use left modules).

By Lemma \ref{lem:left-right} (c,d), we have that $S_1$ is in $\X(\p)$
and $S_2$ is in $\Y(\p)$.
Now the claim follows from Lemma \ref{lem:splits}.
\end{proof}

Summarizing, we have the following
description of $Q_B$ and $J_B$.

\begin{lemma}\label{lem:ac}
The quiver $Q_B$ is acyclic and the ideal $J_B$ is generated by the paths from $V_R$ to $V_L$.
\end{lemma}

\begin{proof}
By Lemma~\ref{lem:here} and \ref{lem:qb-ver1}(b),
there are no cycles in the subquivers $Q_B^{LM}$ and $Q_B^{MR}$.
On the other hand, there are no paths from $V_L$ to $V_R$ by Lemma~\ref{lem:nopath}.
Hence, there are no cycles in $Q_B$.
It follows from Lemma~\ref{lem:qb-ver1}(c) and Lemma~\ref{lem:nopath}
that $J_B$ is generated by the paths from $V_R$ to $V_L$.
\end{proof}

The following facts, see \cite[Theorem~0.5 and Proposition~1.1]{air}, will be crucial.
Recall that a torsion pair $(\T,\F)$ in $\mod \Lambda$ is called functorially finite
if both $\T$ and $\F$ are functorially finite subcategories.

\begin{proposition}\label{prop:air}
Let $(\T,\F)$ be a torsion pair in
$\mod \Lambda$, for a finite dimensional algebra
$\Lambda$.
\begin{itemize}
\item[(a)] $(\T,\F)$ is functorially finite
if and only if there is a 2-term silting complex $\c$
in $K^b(\proj \Lambda)$, such that
$(\T,\F) = (\T(\c),\F(\c))$.
\item[(b)] $(\T,\F)$ is functorially finite
if and only if $\T = \Fac U$, where $U$ is $\Ext$-projective in $\T$.
\end{itemize}
\end{proposition}

We need one more observation before finishing the proof
of Proposition~\ref{prop:key}.

\begin{lemma}\label{lem:xy}
We have
 $\X(\p)  \subset\mod\End_A(H^{-1}(\nu \p))$
 and  $\Y(\p)  \subset\mod\End_A(H^{0}(\p))$.
\end{lemma}

\begin{proof}
We only prove the first inclusion, the second can
be proved dually.
 By the second part of Lemma \ref{lem:surj}, we only need
to prove that for any element $\gamma$ in $\End_{\DA}(\p)$
which factors through $\add \p_R$, we have $X\gamma =0$
for any $X$ in $\X(\p)$. This holds since $\X(\p) = \Hom_{\DA}(\p, \F(\p)[1])$ and clearly $\Hom_{\DA}(\p_R, M[1]) = 0$ for any
 $A$-module $M$.
\end{proof}

\begin{proof}[Proof of Proposition \ref{prop:key}]
Let $\p$ be a 2-term silting complex in $K^b(\proj A)$ such that the induced torsion pair
$(\T(\p),\F(\p))$ satisfies condition (d) in Proposition~\ref{prop:CL} and let $B=\End_\DA(\p)=\k Q_B/J_B$. By Lemma~\ref{lem:ac},  the quiver $Q_B$ is acyclic and the ideal $J_B$ is generated by the paths from the vertices in $V_R$ to the vertices in $V_L$. Let $H=\k Q_B$.  We then have an induced embedding $\mod B\subset\mod H$. We claim that $(\X(\p),\Y(\p))$ is also a torsion pair in $\mod H$.


For any $H$-module $M$, we only need to prove that $M$ is in $\X(\p)\ast\Y(\p)$.
We first note that $Me_LH$ is in $\X(\p)$
by Lemma \ref{lem:left-right}(c). Consider
the short exact sequence
\begin{equation}\label{M-seq}
      0\rightarrow Me_LH\rightarrow M\rightarrow M/Me_LH\rightarrow 0.
\end{equation}
By the description of $J_B$
it is clear that $MJ_B \subset Me_LH$, and hence that
$N= M/Me_LH$ is in $\mod B$.
Now, using that $(\X(\p),\Y(\p))$ is a torsion pair in $\mod B$,
we have $N\in\X(\p)\ast\Y(\p)$.
Since $\X(\p)$ is closed under extensions in $\mod B$,
by Lemma~\ref{lem:xy},
it is also an extension-closed subcategory of $\mod\End_A(H^{-1}(\nu\p))$.
By Lemma~\ref{lem:here}, Lemma~\ref{lem:qb-ver1}(b) and the definition of $H$, we have
$\End_A(H^{-1}(\nu\p))\cong H/He_RH$.
Hence $\X(\p)$ is also closed under extensions in $\mod H$.
Using the sequence \eqref{M-seq},
we have that $M\in\X(\p)\ast\X(\p)\ast\Y(\p)=\X(\p)\ast\Y(\p)$.
Therefore, $(\X(\p),\Y(\p))$ is also a torsion pair in $\mod H$.

We also claim that for any $X\in\X(\p)$ and $Y\in\Y(\p)$, we have a functorial isomorphism $\Ext^1_H(X,Y)\cong\Ext^1_B(X,Y)$.
For this, it is sufficient to prove that for any short exact sequence in $\mod H$:
      \[0\rightarrow Y\rightarrow E\rightarrow X\rightarrow 0,\]
we have that $E\in\mod B$. For any vertex $v_r$ in $V_R$,
we have $Xe_{v_r} = 0$, by Lemma \ref{lem:xy}, and
similarly $Ye_{v_l} = 0$ for any vertex $v_l$ in $V_L$.
For an arbitrary path $e_{v_r} p e_{v_l}$ in $J_B$, we
then have that $X(e_{v_r} p e_{v_l}) = 0$, so
$E(e_{v_r} p e_{v_l}) \subset Y e_{v_l} = 0$.
Hence $EJ_B =0$, and $E$ is in $\mod B$.

The torsion pair $(\X(\p),\Y(\p))$ is a functorially finite torsion pair
in $\mod B$, by \cite[Corollary~3.9]{bz}, and it then follows from Proposition \ref{prop:air}(b) that it is also
functorially finite in $\mod H$.

Hence, by Proposition
\ref{prop:air}(a), it follows that there is a 2-term
silting complex $\rr$ in $K^b(\proj H)$ such that
$(\T(\rr),\F(\rr)) = (\X(\p),\Y(\p))$.
Since $H$ is hereditary, the torsion pair
$(\X(\rr),\Y(\rr))$ is split by Proposition~\ref{prop:basic}(a).
By the silting theorem, we have $\X(\rr) \simeq \F(\rr)=\Y(\p) \simeq \T(\p)$ and
similarly $\Y(\rr) \simeq \F(\p)$.



So we have split torsion pairs $(\T(\p),\F(\p))$ in $\mod A$, and $(\X(\rr),\Y(\rr))$ in $\mod \End_{D^b(H)}(\rr)$.
We claim that we actually have
$\mod A \simeq \mod \End_{D^b(H)}(\rr)$.
For this, we need in addition a functorial isomorphism
$\Hom_{\End_{D^b(H)}(\rr)}(\Y(\rr), \X(\rr)) \cong \Hom_A(\F(\p), \T(\p))$. Indeed, we have
\begin{align*}
\Hom_A(\F(\p), \T(\p)) & \cong \Hom_{\DA}(\F(\p)[1][-1], \T(\p))  \\
& \cong \Ext^1_{B}(\X(\p), \Y(\p)) \\
& \cong \Ext^1_{H}(\X(\p), \Y(\p)) \\
& = \Ext^1_{H}(\T(\rr),\F(\rr)) \\
&\cong \Hom_{D^b(H)}(\T(\rr), \F(\rr)[1]) \\
& \cong \Hom_{\End_{D^b(H)}(\rr)}(\Y(\rr), \X(\rr))
\end{align*}
It now follows that $\mod A \simeq \mod \End_{D^b(H)}(\rr)$.
Hence $A \cong \End_{D^b(H)}(\rr)$, and we have proved
that $A$ is silted.
\end{proof}

The following lemma gives a sufficient condition for
a finite dimensional algebra to be tilted.

\begin{lemma}\label{lem:tilted}
Let $A$ be a finite dimensional algebra and $\p$ a 2-term silting complex, such that the condition of Proposition \ref{prop:key} holds. If, in addition,
 $\T(\p)$ contains all the injective $A$-modules, or $\F(\p)$ contains all the projective $A$-modules, then $A$ is a tilted algebra.
\end{lemma}

\begin{proof}
Assume $\T(\p)$ contains all the injective $A$-modules. Then, we have $\nu \p_L[-1]\in\T(\p)$.
So $0=\Hom_\DA(\p,\nu \p_L[-1][1])=\Hom_\DA(\p,\nu \p_L) \cong
D\Hom_\DA(\p_L,\p)$. In particular, $\Hom_\DA(\p_L,\p_L)=0$.
Hence $\p_L = 0$. Then $|H^{0}(\p)|=|\p|=|A|$. Since $H^0(\p)\in\T(\p)$, we have  $\id H^0(\p)_A\leq 1$ by assumption. It is clear that $\Ext^1_A(H^0(\p),H^0(\p))=0$. So $H^0(\p)$ is a cotilting $A$-module. By Lemma~\ref{lem:here},
the algebra $\End_A(H^0(\p))$ is hereditary. Therefore, the algebra $A$ is tilted. The other case can be proved dually.
\end{proof}

Now we prove the main result in this section.

\begin{theorem}\label{thm:1-silt}
Let $A$ be a connected finite dimensional algebra over
an algebraically closed field $\k$. Then the following are equivalent:
\begin{itemize}
  \item[(a)] $A$ is a silted algebra;
  \item[(b)] there is a split functorially finite torsion pair $(\T,\F)$ in $\mod A$ such that $\id_A X\leq 1$ for any $X\in\T$ and $\pd_A Y\leq1$ for any $Y\in\F$;
  \item[(c)] $A$ is a shod algebra with $(\add(\R\setminus\L),\add\L)$ functorially finite;
  \item[(d)] $A$ is a tilted algebra or a strictly shod algebra.
\end{itemize}
In this case, the global dimension of $A$ is at most 3.
\end{theorem}

\begin{proof}
(a)$\Rightarrow$(b): This follows from combining Proposition \ref{prop:basic}(a,c)
and Proposition~\ref{prop:air}(a).

\noindent (b)$\Rightarrow$(d): Assuming (b), it follows from Proposition~\ref{prop:air}(a) that there is a 2-term silting complex $\p\in K^b(\proj A)$ such that $(\T,\F)=(\T(\p),\F(\p))$. If $\F(\p)$ contains all the projective $A$-modules, then $A$ is a tilted algebra by Lemma~\ref{lem:tilted}. If there is a projective $A$-module in $\T(\p)$, since $\T(\p)$ is contained in $\add\R$, then $\R$ contains an $\Ext$-projective module. By \cite[Theorem~II.3.3]{hrs}, if $A$ is quasi-tilted, then $A$ is tilted. Thus the proof is complete.

\noindent(d)$\Rightarrow$(c): See \cite[Theorem~3.6]{chu}.

\noindent(c)$\Rightarrow$(b): This is trivial.

\noindent (b)$\Rightarrow$(a): This follows from
  combining Propositions \ref{prop:key} and
  \ref{prop:air}.

Finally, recall that it was proved in Proposition \ref{prop:basic}(e),
that the global dimension of a silted algebra is at most 3.
  \end{proof}

Note that we have now proved part (a) of Theorem \ref{main}.

\section{Double sections}\label{sec:double}

In \cite{rs}, Reiten and Skowro\'{n}ski characterized strictly shod algebras as strict {\em double tilted algebras}, which are algebras whose AR-quiver contains a strict faithful double section with certain conditions.
Let $B$ be an algebra which is tilted or strictly shod.
By the previous section, we know that
$B\cong \End_{D^b(H)}(\p)$ for a 2-term silting complex $\p$ in
the bounded derived category of
some hereditary
algebra $H$. In this section,
we will use this fact to give an alternative proof of why
$B$ has a faithful double section $\Delta$, by identifying the
modules in $\Delta$ as images of some injective or projective
$A$-modules under the functors $\Hom_{D^b(H)}(\p,-)$ or $\Hom_{D^b(H)}(\p,-[1])$.
Furthermore, we have that
$\Delta$ is a section when $B$ is tilted, while it is a strict double section when $B$ is strictly shod. The construction of the double section $\Delta$ in a silted algebra is an analogue of the construction of a section in a tilted algebra. For the latter, we refer to \cite[Section~VIII.3]{ass}.

We recall some definitions concerning AR-quivers. For an algebra $\Lambda$, denote by $\Gamma_\Lambda$ the AR-quiver of $\Lambda$ and by $\tau_\Lambda=D\text{Tr}$ and $\tau_\Lambda^{-1}=\text{Tr}D$ the AR-translations in $\Gamma_\Lambda$. A $\tau_\Lambda$-orbit of a module $M\in\mod \Lambda$ is the collection $\{\tau_\Lambda^m M\mid m\in\mathbb{Z}\}$. A path $x_0\rightarrow x_1\rightarrow \cdots \rightarrow x_{s-1}\rightarrow x_{s}$ in $\Gamma_\Lambda$ is called \emph{sectional} if there is no $i$ with $1\leq i\leq s-1$ such that $x_{i-1}=\tau_\Lambda x_{i+1}$ and is called \emph{almost sectional} if there is exactly one $i$ with $1\leq i\leq s-1$ such that $x_{i-1}=\tau_\Lambda x_{i+1}$.

Let $\cc$ be a connected component of $\Gamma_\Lambda$.  A connected full subquiver $\Delta$ of $\cc$ is called a \emph{double section} in $\cc$ if the following conditions hold:
\begin{itemize}
  \item[-] $\Delta$ is acyclic, i.e. there is no oriented cycles in $\Delta$;
  \item[-] $\Delta$ is convex, i.e. for each path $x_1\rightarrow x_2\rightarrow\cdots\rightarrow x_s$ in $\cc$ with $x_1,x_s\in\Delta$ we have $x_i\in\Delta$ for all $1\leq i\leq s$;
  \item[-] For each $\tau_\Lambda$-orbit $\oo$ in $\cc$ we have $1\leq|\Delta\cap\oo|\leq 2$;
  \item[-] If $\oo$ is a $\tau_\Lambda$-orbit in $\cc$ and $|\Delta\cap\oo|= 2$ then $\Delta\cap\oo=\{X,\tau_\Lambda X\}$ for some $X\in\cc$ and there are sectional paths $I\rightarrow\cdots\rightarrow \tau_\Lambda X$ and $X\rightarrow\cdots\rightarrow P$, with $I$ injective and $P$ projective.
\end{itemize}
A double section $\Delta$ in $\cc$ is called {\em strict} if there exists a $\tau_\Lambda$-orbit $\oo$ in $\cc$ with $|\Delta\cap\oo|= 2$ and is called a {\em section} if for any $\tau_\Lambda$-orbit $\oo$ in $\cc$ we have $|\Delta\cap\oo|= 1$. A double section is called
{\em faithful}, if the direct sum of the corresponding
modules is faithful.

Now let $B$ be a connected silted algebra, that is, $B$ is connected and there is a hereditary algebra $H$ and a 2-term silting complex $\p\in K^b(\proj H)$ such that $B=\End_\dh(\p)$. Let $F(-) =\Hom_\dh(\p,-) \colon \dh \rightarrow \mod B$.

Let $\pp$ be a complete set of non-isomorphic indecomposable projective $H$-modules. Let $\pp_l$ be the subset of $\pp$ consisting of $P$ with $P\in\add\p$ and let $\pp_r$ be the subset of $\pp$ consisting of $P$ with $P[1]\in\add\p$. It is clear that $\pp_l\cap\pp_r=\emptyset.$

\begin{lemma}\label{lem:conn}
With the above notation, the following hold. 
\begin{itemize}
  \item[(a)] For any $P\in\pp$, we have $F(P[1])\in\X(\p)$ and $F(\nu P)\in\Y(\p)$; $F(P[1])= 0$ if and only if $P\in\pp_l$; $F(\nu P)= 0$ if and only if $P\in\pp_r$.
  \item[(b)] For any $P\in\pp\setminus\left(\pp_l\cup\pp_r\right)$, we have that both of $F(\nu P)$ and $F(P[1])$ are indecomposable and there is an AR-sequence
  \[0\rightarrow F(\nu P)\rightarrow F(\nu P/\soc(\nu P)) \oplus F((\rad P)[1])\rightarrow F(P[1])\rightarrow0.\]
  In particular, in this case, $F(\nu P)$ is not injective and $F(P[1])$ is not projective.
  \item[(c)] An indecomposable $B$-module in $\X(\p)$ is projective if and only if it is isomorphic to $F(P[1])$ for $P\in\pp_r$. In this case, there is a right minimal almost split map in $\mod B$
      \[F(\nu P/\soc (\nu P))\oplus F((\rad P)[1])\rightarrow F(P[1]).\]
  \item[(d)] An indecomposable $B$-module in $\Y(\p)$ is injective if and only if it is isomorphic to $F(\nu P)$ for $P\in\pp_l$. In this case, there is a left minimal almost split map in $\mod B$
      \[F(\nu P)\rightarrow F(\nu P/\soc(\nu P))\oplus F((\rad P)[1]).\]
\end{itemize}
\end{lemma}

\begin{proof}
Statement (a) follows from Lemma \ref{lem:left-right} and the definitions of $\pp_l$ and $\pp_r$ and (b) follows from Proposition~\ref{prop:basic}(f). We will prove (c). Statement (d) can be proved similarly. For any $P\in\pp_r$, we have $P[1]\in\add\p$. So $F(P[1])$ is projective. Now we prove that all indecomposable projective modules in $\X(\p)$ are of this from. Let $F(M[1])$ be an indecomposable projective $B$-module in $\X(\p)$ with $M\in\F(\p)$. Applying the functor $F(-)$ to the projective cover $P_M\xrightarrow{p_M} M$ of $M$ in $\mod H$, we obtain an epimorphism $F(P_M[1])\to F(M[1])$ since $F(\ker p_M[2])=0$. Since $F(M[1])$ is projective, this epimorphism is split. Hence $F(M[1])$ is a direct summand of $F(P_M[1])$. Therefore, by (a) and (b), $F(M[1])$ has to have the form $F(P[1])$ for some $P\in\pp_r$.

By \cite[Chap. 4]{h}, there is an AR-triangle
\[\nu P\rightarrow \nu P/S \oplus(\rad P)[1]\rightarrow P[1]\rightarrow (\nu P)[1] \]
in $\dh$, where $S=\soc(\nu P)$.
Applying the functor $F$ to this triangle,
we obtain an exact sequence
 \[0\rightarrow F(\nu P/S)\oplus F(\rad P[1])\rightarrow F(P[1])
 \xrightarrow{u} F((\nu P)[1])\]
 in $\mod B$.
 Note that the last map $u$ in this exact sequence factors through $F(S[1])$.
For each indecomposable
 summand $\p'$ of $\p$, if $\p'$ is of the form $P'[1]$ for some $P'\in\pp_r$, then \sloppy $\Hom_\dh(\p',S[1])$ is 1-dimensional for $P'\cong P$, and 0-dimensional for $P'\ncong P$. If $\p'$ is not of such form, then
 $H^{-1}(\p') = 0$ and hence $\Hom_\dh(\p',\nu P[1])\cong D\Hom_\dh(P[1],\p')=0$. So the image of $u$ is the simple top of $F(P[1])$. Hence $F(\nu P/S)\oplus F(\rad P[1])\rightarrow F(P[1])$ is a right minimal almost split map.

\end{proof}

Let $\pp'_r$ be the subset of $\pp$ consisting of modules from which there are nonzero morphisms to modules in $\pp_r$ which do not factor through modules in $\pp_l$. Dually, let $\pp'_l$ be the subset of $\pp$ consisting of modules to which there are nonzero morphisms from modules in $\pp_l$ which do not factor through modules in $\pp_r$.


\begin{theorem}\label{thm:sec}
Let $H$ be a finite dimensional hereditary algebra and $\p$ be a 2-term silting complex in $K^b(\proj H)$ such that $B=\End_\dh(\p)$ is connected. Then the full subquiver $\Delta$ of $\Gamma_B$ formed by $F(P[1])$ for $P\in\pp'_r$ and by $F(\nu P)$ for $P\in\pp'_l\cup\left(\pp\setminus\pp'_r\right)$, is a faithful double section in a component $\Psi_\p$ of $\Gamma_B$. Moreover, $\Delta$ is a section if and only if $B$ is tilted, while $\Delta$ is a strict double section if and only if $B$ is strictly shod.
\end{theorem}

\begin{proof}
We first note that, since $A$ is hereditary,
then for a projective $P$ with
$P/\rad P \cong S$, we
have that $\rad P$ is projective and $\nu P/S$ is injective.
By Lemma~\ref{lem:conn}, $\pp'_r$ is the set of modules $P\in\pp$ such that there is a path in $\Gamma_B$ from $F(P[1])$ to $F(P'[1])$ for some $P'\in\pp_r$, and $\pp'_l$ is the set of $P\in\pp$ such that there is a path from $F(\nu P')$ to $F(\nu P)$ for some $P'\in\pp_l$.

Let $Z_1\rightarrow Z_2\rightarrow\cdots\rightarrow Z_s$ be a path in $\Gamma_B$ with $Z_1,Z_s\in\Delta$. To prove that $\Delta$ is convex, it is clearly sufficient to prove that $Z_2$ is in
$\Delta$, and proceed by induction.


By assumption, the module $Z_1$ is either of the form $F(P[1])$ 
for $P \in \pp'_r$, or
of the form $F(\nu P)$ for $P\in\pp_l' \cup (\pp \setminus \pp_r')$.

First assume that $Z_1$ is of the form $F(P[1])$. Since $F(P[1])$ belongs to $\X(\p)$, by Proposition~\ref{prop:basic}(b) we
have that none of the $Z_2, \dots Z_s$ are in $\Y(\p)$, and hence none
are of the form $F(\nu Q)$ for a projective $Q$. In particular
$Z_s =F(P'[1])$ for some projective $P'$ in $\pp'_r$, and
using repeatedly Lemma~\ref{lem:conn}(b,c) also $Z_{s-1}, Z_{s-2},
\dots, Z_2$ must have this
property. Hence, the claim that $Z_2$ is in $\Delta$ holds in this case.

Now assume $Z_1$ is of the form $F(\nu P)$ for some
$P$ in $\pp'_l\cup\left(\pp\setminus\pp'_r\right)$.
By Lemma~\ref{lem:conn}(b, d) we then have that either
$Z_2 = F(Q[1])$ or $Z_2 = F(\nu Q)$ for a projective $Q$.
In case $Z_2 = F(Q[1])$, we can use the argument for case I, to conclude
that $Q$ is in $\pp'_r$ and hence $Z_2$ is in $\Delta$. Therefore
assume $Z_2 = F(\nu Q)$. Then by Lemma~\ref{lem:conn}(b,d), the map $Z_1\to Z_2$ is induced by an irreducible map $P\to Q$. If $P\in\pp'_l$, then $Q\in\pp'_l$ and we are done. If $P\notin\pp'_l$, then
since $P$ is by assumption not in $\pp'_r$, we must also have
that $Q$ is not in $\pp'_r$. This finishes the proof for the claim that $Z_2$ is in $\Delta$ for case II. Hence, we have that $\Delta$ is convex.


We next prove that $\Delta$ is acyclic.
Let $\Delta'$ be the full subquiver of $\Gamma_B$ formed by $F(P[1])$ for $P\in\pp\setminus\pp_l$ and by $F(\nu P)$ for $P\in\pp\setminus\pp_r$. It
follows from Lemma~\ref{lem:conn} that $\Delta'$ is convex and acyclic. It is clear that $\Delta$ is a full subquiver of $\Delta'$. So $\Delta$ is also acyclic.

We proceed to show that $\Delta$ is faithful and connected. For this,
consider the $B$-modules $T_a =\oplus_{P\in\pp_r}F(P[1])\in\X(\p)$ and $T_b = \oplus_{P\in\pp\setminus\pp_r}F(\nu P)\in\Y(\p)$. We claim that $T = T_a \oplus T_b$ is a tilting module. Indeed, by Proposition~\ref{prop:basic}(d) we have $\pd T_b \leq 1$, and by
Lemma~\ref{lem:conn}(c) it follows
that $T_a$ is projective. So we have that $\Ext^1(T_b,T_a)
\cong D\Hom(T_a, \tau_B T_b)$.
Since by Proposition~\ref{prop:basic}(b) we have that
$\Y(\p)$ is closed under predecessors, we must
have that $\tau_B T_b$ is also in
$\Y(\p)$. Hence, since $T_a$ is in $\X(\p)$,
we have that $\Ext^1(T_b,T_a)\cong D\Hom_B(T_a, \tau_B T_b) = 0$.
The $B$-module $F(\nu \p)$ is $\Ext$-injective
in $\Y(\p)$ by \cite[Proposition 2.8(3)]{bz}. Hence, we 
have $\Ext^1_B(T_b,T_b) = 0$. Hence $\Ext^1_B(T,T) = 0$.
Since clearly $|T|=|A|=|B|$, we have that $T$ is a tilting $B$-module.

Now, let $\Delta_T$ be the smallest convex full subquiver of $\Gamma_B$ which contains all indecomposable summands of $T$.
Since $T$ is a tilting module, we have that $\Delta_T$ is connected and faithful. It is easy to check that $\Delta_T$ is the full subquiver of $\Delta'$ formed by $F(\nu P)$ for $P\in\pp\setminus\pp_r$ and by $F(P[1])$ for $P\in\pp'_r$. So $\Delta$ is a full subquiver of $\Delta_T$ and $\Delta_T\setminus\Delta$ is contained in $\Y(\p)\cap\Delta'$.



We will construct recursively a sequence of faithful connected full subquivers $\Delta_0=\Delta_T$, $\Delta_1$, $\cdots$, $\Delta_m=\Delta$ of $\Delta'$ such that all of them contain $\Delta$ as a full subquiver and $\Delta_{s+1}$ is a full subquiver of $\Delta_{s}$ with one less vertex for each $0\leq s\leq m-1$. Assume that $\Delta_s$ has been constructed for some $s$. By assumption we have $\Delta\subset\Delta_{s}\subset\Delta_T$ and so $\Delta_s\setminus\Delta\subset\Y(\p)\cap\Delta'$. Then for each vertex $Z=F(\nu P)$ in $\Delta_{s} \setminus \Delta$, there is no path from $F(\nu P')$ to $Z$ for any $P'\in\pp_l$, but there is a path from $\tau_B^{-1}Z = F(P[1])$ to $F(P''[1])$ for some $P''\in\pp_r$. So $\tau_B^{-1}Z\in\Delta_s$ and one can choose a vertex $Z\in\Delta_{s} \setminus \Delta$ which is a source in $\Delta_s$. Now let $\Delta_{s+1}$ be the full subquiver of $\Delta'$ obtained from $\Delta_s$ by removing $Z$ and the arrows adjacent to  $Z$. By Lemma~\ref{lem:conn}(b), we have that $\Delta_{s+1}$ is also faithful and connected. This finishes the construction and the
proof that $\Delta$ is faithful and connected.

Now let $\Psi_\p$ be the connected component of $\Gamma_B$ which contains $\Delta$.
By the construction of $\Delta$,
if a $\tau_B$-orbit $\oo$ in $\Psi_\p$ intersects $\Delta$,
then $|\oo\cap\Delta|\leq 2$.
When it equals $2$,
the last condition in the definition of double section holds.
So what we need to prove is that $\Delta$ intersects each $\tau_B$-orbit in $\Psi_\p$.
By definition, for each $P\in\pp$,
at least one of $F(\nu P)$ and $F(P[1])$ is in $\Delta$.
Hence what we need to prove is equivalent to that $\Delta'$ intersects each $\tau_B$-orbit in $\Psi_\p$.
This proof is similar to the proof for the tilting case (cf. e.g. the proof of \cite[Theorem~VIII.3.5]{ass}), but we provide details for completeness.
By induction, we only need to prove that for any $\tau_B$-orbit $\oo\subset\Psi_\p$, if there is an arrow $\tau^nY\to Z$ or $Z\to\tau^nY$ in $\Psi_\p$ for some $n\in\mathbb{Z}$, a module $Z\in\oo$ and a module $Y\in\Delta'$, then $\oo$ intersects $\Delta'$. We assume that $|n|$ is minimal, and consider the following three cases.

\begin{itemize}
  \item[-] The case $n<0$. We first claim that
  $Y=F(P[1])$ for some $P\in\pp$, since we otherwise can replace $Y$ by $\tau^{-1} Y$, and then this contradicts the minimality of $|n|$. We next claim that it follows that $Z\in\Delta'$ and then we are done. To prove this claim, assume first
  that $Z$ is in $\X(\p)$ but not in $\Delta'$. Then, by Lemma~\ref{lem:conn}(c), it is not projective. So $\tau Z\neq 0$ and there exists an arrow $\tau Z\to\tau^{n+1}Y $ or $\tau^{n+1}Y\to \tau Z$. This contradicts the minimality of $|n|$. Now assume $Z$ is in $\Y(\p)$ but not in $\Delta'$. Then, by Lemma~\ref{lem:conn}(d), it is not injective and then $\tau^{-1}Z\neq 0$. There is no arrow from $\tau^n Y$ to $Z$ since $\tau^n Y\in\X(\p)$. If there is an arrow  $Z\to \tau^n Y$, then $\tau^{-1}Z\in \X(\p)$ since it is a successor of $\tau^n Y$. Since $Z\in\Y(\p)$, the AR-sequence starting at $Z$ is a connecting sequence, which implies that $Z\in\Delta'$. This is a contradiction.
  \item[-] The case $n>0$. This is dual to the above case.
  \item[-] The case $n=0$. If there exists an arrow $Y\to Z$ with $Y=F(P[1])$ for some $P\in\pp$, then $Z\in\X(\p)$. If $Z$ is projective, then it is in $\Delta'$; if $Z$ is not projective, then $\tau Z\neq 0$ and there is an arrow $\tau Z\to Y$, which implies $\tau Z\in\Delta'$ by Lemma~\ref{lem:conn}. If there exists an arrow $Y\to Z$ with $Y=F(\nu P)$ for some $P\in\pp$, the claim follows directly
  from Lemma~\ref{lem:conn}. Similarly, for the case $Z\to Y$, we also have that $\oo$ intersects $\Delta'$.
\end{itemize}
Therefore, we have proved that $\Delta$ is a faithful double section.

To proceed, consider the following full subquivers of $\Delta'$
\[\Delta'_l=\{x\in\Delta\mid \text{there is an almost sectional path $x\to\cdots\to F(P[1])$ for some $P\in\pp_r$}\},\]
\[\Delta'_r=\{y\in\Delta\mid \text{there is an almost sectional path $F(\nu P)\to\cdots\to y$ for some $P\in\pp_l$}\},\]
\[\Delta_l=\left(\Delta\setminus\Delta'_r\right)\cup\tau_B\Delta'_r,\]
\[\Delta_r=\left(\Delta\setminus\Delta'_l\right)\cup\tau_B^{-1}\Delta'_l.\]
We claim that for any modules $X$ from $\Delta_r$ and $Y$ from $\Delta_l$, we have $\Hom_B(X,\tau_B Y)=0$. By definition, we have $\{F(P[1])\mid P\in\pp'_l\cap\pp'_r\}\subset\Delta'_r$
and $\{F(\nu P)\mid P\in\pp'_l\cap\pp'_r\}\subset\Delta'_l$.
So \[\begin{array}{ccccc}
\Delta_l\cap\X(\p)&=&\left(\Delta\setminus\Delta'_r\right)\cap\X(\p)&\subset&\{F(P[1])\mid P\in\pp'_r\setminus\left(\pp'_l\cap\pp'_r\right)\},\\
\Delta_l\cap\Y(\p)&\subset&\Delta'\cap\Y(\p)&=&\{F(\nu P)\mid P\in\pp\setminus\pp_r \},\\
\Delta_r\cap\Y(\p)&=&\left(\Delta\setminus\Delta'_l\right)\cap\Y(\p)&\subset&\{F(\nu P)\mid P\in\pp\setminus\pp'_r\}.
\end{array}\]
Hence by Lemma~\ref{lem:conn}, we have $\tau_B\left(\Delta_l\cap\X(\p)\right)\subset\{F(\nu P)\mid P\in\pp'_r\setminus\left(\pp'_l\cap\pp'_r\right)\}$. First, recall that we have proved that $T_b=\oplus_{P\in\pp}F(\nu P)$ is a direct summand of a tilting $B$-module. Then $\Hom_B(T_b,\tau_BT_b)=0$. It follows that
\begin{equation}\label{eq:31}
\Hom_B(\Delta_r\cap\Y(\p),\tau_B\left(\Delta_l\cap\Y(\p)\right))=0.
\end{equation}
Second, for any map $f$ from $P_1\in\pp\setminus\pp'_r$ to $P_2\in\pp'_r\setminus\left(\pp'_l\cap\pp'_r\right)$, since $P_2\in\pp'_r$ but $P_1\notin\pp'_r$, we have that $f$ factors through $\pp_l$ as $f_1f_2$. Furthermore, since $P_2\notin\pp'_l$, the map $f_1$ factors through $\pp_r$. Hence $F(\nu f)=0$. By the silting theorem, it is easy to check that $F(\nu-)$ induces an epimorphism $\Hom_A(P_1,P_2)\to\Hom_B(F(\nu P_1),F(\nu P_2))$. Therefore, we have that
\begin{equation}\label{eq:32}
\Hom_B\left(\Delta_r\cap\Y(\p),\tau_B\left(\Delta_l\cap\X(\p)\right)\right)=0.
\end{equation}
Third, note that $\tau_B\Delta_l\subset\Y(\p)$, so we have that
\begin{equation}\label{eq:33}
\Hom_B\left(\Delta_r\cap\X(\p),\tau_B\Delta_l\right)=0.
\end{equation}
Combining the equations \eqref{eq:31}, \eqref{eq:32} and \eqref{eq:33}, we complete the proof of the claim.

It now follows that
if $\Delta$ is a double section,
then $B$ is a strictly shod algebra by \cite[Theorem~8.2]{rs}.
On the other hand, if $\Delta$ is a section, which implies that $\Delta_l=\Delta_r=\Delta$, then by \cite[Theorem~1.6]{l} and \cite[Theorem~3]{sko}, it follows that $B$ is a tilted algebra. By Theorem~\ref{thm:1-silt},
the algebra $B$ is either tilted or strictly shod, therefore we have the last assertion.
\end{proof}


\begin{remark}
In the above proof, a tilting module is constructed for each functorially finite torsion pair $(\T,\F)$ satisfying condition (d) in Proposition~\ref{prop:CL} in a silted algebra, which is the one considered in \cite{cl} when $(\T,\F)=(\add(\R\setminus\L),\add\L)$.
\end{remark}

\begin{remark}
In general, $B$ is not necessarily connected even if $H$ is connected. In this case, the subquiver $\Delta$ constructed in the above theorem is a union of faithful double sections $\Delta_i$ in components $\Psi_\p^{i}$ of $\Gamma_{B_i}$, where each $B_i$ is a connected component of $B$.
\end{remark}

\begin{corollary}\label{cor:strict}
Let $H$ be a finite dimensional hereditary algebra and $\p$ be a 2-term silting complex in $K^b(\proj H)$ such that $B=\End_\dh(\p)$ is connected. Then $B$ is strictly shod if and only if there are nonzero morphisms $f \colon P_1\to P_2$ and $g \colon P_2\to P_3$ with $P_1\in\pp_l$, with  $P_2\in\pp\setminus(\pp_l\cap\pp_r)$ and with $P_3\in\pp_r$ such that $f$ does not factor through $\pp_r$ and $g$ does not factor through $\pp_l$.
\end{corollary}

\begin{proof}
By Theorem~\ref{thm:sec}, the algebra $B$ is strictly shod if and only if $\Delta$ is a strict double section, that is, there is a $P_2\in\pp\setminus\left(\pp_l\cup\pp_r\right)$ such that both $F(\nu P)$ and $F(P[1])$ belong to $\Delta$. By the construction of $\Delta$, this is equivalent to that $P_2\in\pp'_l\cap\pp'_r$. Then by definition, this is equivalent to that there is a morphism from a module $P_1\in\pp_l$ to $P_2$, which does not factor through $\pp_r$ and there is a morphism from $P_2$ to $P_3\in\pp_r$, which does not factor through $\pp_l.$ Thus, the proof is complete.
\end{proof}

\begin{corollary}\label{cor:tilted}
Let $H$ be a finite dimensional hereditary algebra and $\p$ be a 2-term tilting complex. Then $\End_\dh (\p)$ is a tilted algebra.
\end{corollary}

\begin{proof}
This follows from Theorem~\ref{thm:1-silt} and Corollary~\ref{cor:strict}, using the fact that $\Hom_H(\pp_l,\pp_r)=0$ when $\p$ is tilting.
\end{proof}


%

\section{Abelian hereditary categories}\label{sec:abel}

In this section, we define and study 2-term silting complexes in
bounded derived categories of abelian categories.
Let $\A$ be an abelian $\k$-category.
Assume that $A$ is Ext-finite, i.e., for any objects $M,N\in\A$, we have that $\dim_\k\Ext^i_\A(M,N)$ is finite for all $i\geq0$.
Then $\DAA$ is Krull-Schmidt and Hom-finite (cf. \cite[Section~I.4]{hrs}).
In \cite{hkm}, the authors also study 2-term silting complexes in bounded derived categories of abelian categories.
We remark that the difference between our setting and \cite{hkm} is that we assume that $\A$ is Ext-finite while they assume that $\A$ admits arbitrary coproducts.

\begin{definition}\label{def:silting}
A complex $\p$ in $\DAA$ is called a \emph{2-term silting complex} if the following hold:
\begin{itemize}
  \item[(S1)] $\Hom_\DAA(\p,M[i])=0$ for any $M\in\A$ and $i\neq 0$ or $1$.
  \item[(S2)] $\Hom_\DAA(\p,\p[1])=0$.
  \item[(S3)] For any $M\in\A$, if $\Hom_\DAA(\p,M[i])=0$ for any $i\in\mathbb{Z}$, then $M =0$.
\end{itemize}
\end{definition}

\begin{remark}
When $\A$ is the module category of a finite dimensional $\k$-algebra $A$, we show in Corollary~\ref{cor:alg} that the 2-term silting complexes defined here are the same as the 2-term silting complexes in $K^b(\proj A)$ considered 
in the previous sections of this paper.
\end{remark}

We need the following well-known result concerning truncations in $\DAA$.

\begin{lemma}\label{lem:triangles}
Let $\x \in \DAA$ be an object with $H^i(\x) = 0$ for $i<m$ and $i>n$.
Then there exist triangles
\[\x_{i-1}\rightarrow\x_i\rightarrow H^i(\x)[-i]\rightarrow \x_{i-1}[1]\]
for $i\in\mathbb{Z}$ such that $\x_i=0$ for $i<m$ and $\x_n=\x$, where $H^i(\x)$ is the $n$-th cohomology of $\x$.
\end{lemma}

We have the following immediate consequences.

\begin{lemma}\label{lem:homology}
Let $\p$ be a complex in $\DAA$ satisfying (S1). For any $\x\in\DAA$ and $i\in\mathbb{Z}$, there exists an exact sequence
\[0\rightarrow \Hom_\DAA(\p,H^{i-1}(\x)[1])\rightarrow \Hom_\DAA(\p,\x[i])\rightarrow \Hom_\DAA(\p,H^i(\x))\rightarrow0.\]
\end{lemma}

\begin{proof}
Applying $\Hom_\DAA(\p,-)$ to the triangles in Lemma~\ref{lem:triangles}, by (S1), we have
\begin{align}
\Hom_\DAA(\p,\x[i]) & \cong\Hom_\DAA(\p,\x_i[i]) \label{eq:1} \\
\Hom_\DAA(\p,\x_{i-2}[i]) & =0 \label{eq:2} \\
\Hom_\DAA(\p,\x_{i-2}[i+1]) & \cong\Hom_\DAA(\p,\x_{i-1}[i+1])=0 \label{eq:3}
\end{align}
and two exact sequences
\begin{multline*}
\Hom_\DAA(\p,\x_{i-2}[i])\rightarrow\Hom_\DAA(\p,\x_{i-1}[i])\rightarrow\Hom_\DAA(\p,H^{i-1}(\x)[-(i-1)][i])\\\rightarrow\Hom_\DAA(\p,\x_{i-2}[i+1])
\end{multline*}
and
\begin{multline*}
0\rightarrow\Hom_\DAA(\p,\x_{i-1}[i])\rightarrow\Hom_\DAA(\p,\x_{i}[i])\rightarrow\Hom_\DAA(\p,H^{i}(\x)[-i][i])\\\rightarrow\Hom_\DAA(\p,\x_{i-1}[i+1]).
\end{multline*}
By the first exact sequence, together with \eqref{eq:2} and \eqref{eq:3}, we have $\Hom_\DAA(\p,\x_{i-1}[i])\cong\Hom_\DAA(\p,H^{i-1}(\x)[1])$. Then by the second exact sequence, together with \eqref{eq:1} and \eqref{eq:3}, we get the required exact sequence.
\end{proof}

\begin{lemma}\label{lem:pre}
Let $\p$ be a complex in $\DAA$ satisfying (S1).
Then the following hold.
\begin{itemize}
  \item[(a)] $H^i(\p)=0$ for any $i>0$ or $i<-1$.
  \item[(b)] $\Hom_\DAA(\p,\p[i])=0$ for any $i>1$ or $i<-1$.
  \item[(c)] There is a triangle in $\DAA$
             \begin{equation}\label{eq:t1}
             H^{-1}(\p)[1]\to \p\to H^0(\p)\to H^{-1}(\p)[2].
             \end{equation}
  \item[(d)] For any $M\in\A$, the triangle \eqref{eq:t1} induces a functorial isomorphism
             \begin{equation}\label{eq:iso}
             \Hom_\DAA(H^0(\p),M)\cong\Hom_\DAA(\p,M),
             \end{equation}
             and a monomorphism
             \begin{equation}\label{eq:mono}
             \Hom_\DAA(H^0(\p),M[1])\hookrightarrow\Hom_\DAA(\p ,M[1]).
             \end{equation}
\end{itemize}
\end{lemma}

\begin{proof}
Let $n$ be the maximal number such that $H^n(\p)\neq0$ and
let $m$ be the minimal number such that $H^m(\p)\neq0$.
Then, on the one hand, there is a nonzero map $\p\to H^n(\p)[-n]$
by Lemma~\ref{lem:triangles}.
On the other hand, let $\p$ be of the form $\cdots\to P^i\xrightarrow{d^i} P^{i+1}\to\cdots$, then there is a nonzero map from $\p$ to $(P^m/\im d^{m-1})[m]$. Hence, by (S1), we have that $n\leq 0$ and $m\geq-1$. Thus we have assertion (a). Then (b) follows from Lemma \ref{lem:homology} and (c) follows from Lemma~\ref{lem:triangles}. Finally, applying $\Hom_\DAA(-,M)$ to the triangle \eqref{eq:t1} yields the functorial isomorphism \eqref{eq:iso} and the monomorphism \eqref{eq:mono}.
\end{proof}

Let $\p$ be a complex in $\DAA$ satisfying (S1). For an integer $m$, consider the pair of subcategories
\[D^{\leq m}(\p) = \{\x\in \DAA\mid \Hom_{\DAA}(\p,\x[i])=0\text{, for any $i>m$}\}\]
\noindent and
\[D^{\geq m}(\p) = \{\x\in \DAA \mid \Hom_{\DAA}(\p,\x[i])=0\text{, for any $i<m$}\}\]
\noindent in the derived category $\DAA$. Let $\T(\p) = D^{\leq 0}(\p) \cap \A$ and $\F(\p) =  D^{\geq 1}(\p) \cap \A$. Then by (S1), we have
\[\T(\p) = \{X\in \A\mid \Hom_{\DAA}(\p,X[1])=0\}\]
\noindent and
\[\F(\p) = \{X\in \A \mid \Hom_{\DAA}(\p,X)=0\}.\]

We now obtain results similar to \cite[Theorem~2.10]{hkm}.

\begin{lemma}\label{lem:tor}
Let $\p$ be a complex in $\DAA$ satisfying (S1). Then the following hold.
\begin{itemize}
  \item[(a)] $\T(\p)$ is closed under factor objects and $\F(\p)$ is closed under subobjects.
  \item[(b)] $\p$ satisfies (S2) if and only if $H^0(\p)\in\T(\p)$.
  \item[(c)] $\p$ satisfies (S3) if and only if $\T(\p)\cap\F(\p)=\{0\}$.
  \item[(d)] If $\p$ satisfies (S2) and (S3), then:
  \begin{itemize}
  \item[(i)] for each $M\in\A$, there is an exact sequence
             \[
             0\rightarrow tM\rightarrow M\rightarrow M/tM\rightarrow 0
             \]
             with $tM\in\T(\p)$ and $M/tM\in\F(\p)$;
 	\item[(ii)] $(\T(\p), \F(\p))$ is a torsion pair in $\A$;
    \item[(iii)]  $\T(\p)= \Fac H^0(\p)$;
    \item[(iv)] an object $M\in\T(\p)$ is $\Ext$-projective if and only if $M\in\add H^0(\p)$.
\end{itemize}
\end{itemize}
\end{lemma}

\begin{proof}
For (a), we only prove the assertion for $\T(\p)$. The proof for $\F(\p)$ is similar. For any $M\in\T(\p)$ and an epimorphism $f \colon M\to M'$ in $\A$, apply $\Hom_\DAA(\p,-)$ to the exact sequence $0\to\ker f\to M\to M'\to 0$. Since $\Hom_\DAA(\p,\ker f[2])=0$ by (S1), we have that $\Hom_\DAA(\p,M'[1])=0$. So $M'\in\T(\p)$. Then $\T(\p)$ is closed under factor objects.

By Lemma~\ref{lem:homology}, there is an exact sequence
\[0\to\Hom_\DAA(\p,H^0(\p)[1])\to\Hom_\DAA(\p,\p[1])\to\Hom_\DAA(\p,H^1(\p))\to 0.\]
Because $H^1(\p)=0$ by Lemma~\ref{lem:pre}(a), the assertion (b) follows.

For any $M\in\A$, by definition, $M\in\T(\p)\cap\F(\p)$ if and only if $\Hom_\DAA(\p,M[i])=0$ for any $i\in\mathbb{Z}$. So we have (c).

For (d), consider an arbitrary $M\in\A$.
Since $\A$ is Hom-finite and Krull-Schmidt, 
there is a right $\add H^0(\p)$-approximation $g \colon X\to M$. 
Let $tM =\im g\in\Fac H^0(\p)$ and 
consider the exact sequence $0\to tM\to M\to M/tM\to 0$. 
Since $H^0(\p)$ is in $\T(\p)$ by (b), 
it follows from (a) that so is $tM$. 
On the other hand, by Lemma~\ref{lem:pre}(d), 
each map from $\p$ to $M$ factors through $H^0(\p)$, 
hence it factors through $g$. 
So $\Hom_\DAA(\p,M/tM)=0$ and 
then by definition we have $M/tM\in\F(\p)$. Hence (i) holds, and
(ii) follows by definition.
By (a) and (b), we have $\Fac H^0(\p)\subset\T(\p)$. Let $M\in\T(\p)$. Since $(\T(\p), \F(\p))$ is a torsion pair, we have $M\cong tM\in\Fac H^0(\p)$. Hence we have (iii). For (iv), the proof of \cite[Proposition~2.8(2)]{bz} works here, using (iii) and Lemma~\ref{lem:pre}(d).
\end{proof}

Applying results from \cite{hrs}, we obtain that 2-term silting
complexes induce t-structures also in our setting.

\begin{proposition}\label{prop:general}
Let $\A$ be an $\Ext$-finite abelian category, let $\p$ be a 2-term silting complex in $\DAA$ and let $B=\End_\DAA(\p)$. Then the following hold.
\begin{itemize}
   \item[(a)] $(D^{\leq 0}(\p),D^{\geq 0}(\p))$ is a t-structure in $\DAA$.
   \item[(b)] This t-structure is bounded, in the sense that for any $\x\in D^b(\A)$,
   there is an $s$ such that $\x\in D^{\leq s}(\p)$.
  \item[(c)] $\C(\p) \colon =D^{\leq 0}(\p)\cap D^{\geq 0}(\p)$ is an abelian category, where the short exact sequences are the triangles in $\DAA$ whose terms are in $\A$.
  \item[(d)] $\C(\p)=\{\x\in \DAA\mid H^{0}(\x)\in\T(\p),\ H^{-1}(\x)\in\F(\p)\text{ and }H^i(\x)=0\text{ for $i\neq-1$ or $0$}\}$.
  \item[(e)] $(\F(\p)[1],\T(\p))$ is a torsion pair in $\C(\p)$.
  \item[(f)] $\Hom_\DAA(\p,-)$ gives an equivalence from $\C(\p)$ to $\mod B$.
\end{itemize}
\end{proposition}

\begin{proof}
By Lemma~\ref{lem:homology}, we have that
\[D^{\leq 0}(\p)=\{\x\in \DAA\mid \Hom_{\DAA}(\p,H^{i-1}(\x)[1])=\Hom_{\DAA}(\p,H^{i}(\x))=0\text{, for any $i>0$}\}.\]
Then by (S1) and (S3), we have
\begin{equation}\label{eq:neg-part} D^{\leq 0}(\p)=\{\x\in \DAA\mid H^i(\x)=0\text{ for any $i>0$, and }H^0(\x)\in\T(\p)\}.\end{equation}
Similarly, we have
\[D^{\geq 0}(\p)=\{\x\in \DAA\mid H^i(\x)=0\text{ for any $i<-1$, and }H^{-1}(\x)\in\F(\p)\}.\]
Hence by \cite[Proposition~I.2.1, Corollary~I.2.2]{hrs},
(a), (c), (d) and (e) follow. 
To prove (b), consider equation \eqref{eq:neg-part}.
Combined with Lemma~\ref{lem:triangles}, this gives
\begin{equation}\label{eq:resolution}
D^{\leq0}(\p)=\bigcup_{z>0}\A[z]\ast\cdots \ast\A[1]\ast\T(\p).
\end{equation}
Let $\x$ be an arbitrary object in $\DAA$. By  Lemma~\ref{lem:triangles}, it follows that $\x \in
\A[-m]\ast \cdots \ast \A[-n]$ for integers $m<n$.
Then $\x[n+1]$ is in $\A[-m+n+1] \ast \cdots\ast \A[1]$, and
hence $\x \in D^{\leq (n+1)}(\p)$. This proves (b).
For (f), we refer to the proof of \cite[Proposition~3.13]{iy}.
\end{proof}

Recall from \cite{hrs} that an object in $\A$ is called a {\em tilting object} if there exists a torsion pair $(\T,\F)$ in $\A$ satisfying the following properties.
\begin{itemize}
  \item[(T1)] $\T$ is a tilting torsion class, that is, $\T$ is a cogenerator for $\A$.
  \item[(T2)] $\T=\Fac T$.
  \item[(T3)] $\Ext_\A^i(T,X)=0$ for $X\in\T$ and $i>0$.
  \item[(T4)] If $Z\in\T$ satisfies $\Ext_\A^i(Z,X)=0$ for all $X\in\T$ and $i>0$, then $Z\in\add T$.
  \item[(T5)] If $\Ext_\A^i(T,X)=0$ for $i\geq 0$ and $X$ in $\A$ then $X=0$.
\end{itemize}
The following result gives the relationship between 2-term silting complexes and tilting objects.

\begin{proposition}\label{prop:back}
Let $T$ be an object in an $\Ext$-finite abelian category $\A$. Then $T$ is a tilting object in $\A$ if and only if it is a 2-term silting complex in $\DAA$.
\end{proposition}

\begin{proof}
First assume that $T$ is a tilting object in $\A$. Then (S1) follows from \cite[Lemma~4.1]{hrs}, while (S2) follows from (T2) and (T3) and (S3) follows from (T5). So $T$ is a 2-term silting complex in $\DAA$.

Now assume that $T\in\A$ is a 2-term silting complex in $\DAA$. In this case $T=H^0(T)$ is a projective generator in $\C(T)$ by Proposition
\ref{prop:general}(f). Let $\T=\T(T)$. Then (T2) follows from Lemma~\ref{lem:tor}(d.iii); (T3) and (T4) follows from (S1) and Lemma~\ref{lem:tor}(d.iv); (T5) follows from (S3). Now we prove (T1). For any $M\in\A$, consider the canonical exact sequence
\begin{equation}\label{eq:t5}
0\to tM\to M\to M/tM\to 0
\end{equation}
with respect to the torsion pair  $(\T(T),\F(T))$.
Since $T$ is a projective generator in $\C(T)$, there is an exact sequence
\begin{equation}\label{eq:t4}
0\rightarrow N\rightarrow T'\rightarrow (M/tM)[1]\rightarrow 0
\end{equation}
in $\C(T)$ with $T'\in\add T$. Since $T'\in\T(T)$ and $\T(T)$ is closed under subobjects in $\C(T)$, we have $N$ is also in $\T(T)$. The exact sequences \eqref{eq:t5} and \eqref{eq:t4} induce triangles
\[tM\to M\to M/tM\to tM[1]\]
and
\[T'[-1]\rightarrow M/tM\rightarrow N\rightarrow T'\]
in $D^b(\A)$. Because $\Hom_\DAA(T'[-1],(tM)[1])\cong \Hom_\DAA(T',(tM)[2])=0$ by (S1), we have that
the map $T'[-1] \to M/tM$ factors as $T'[-1] \to M \to M/tM$.
Hence, by the octahedral axiom, we have the following commutative diagram of triangles:
\[\xymatrix{
&T'[-1]\ar@{=}[r]\ar[d]&T'[-1]\ar[d]\\
tM\ar[r]\ar@{=}[d]&M\ar[r]\ar[d]&M/tM\ar[r]\ar[d]&tM[1]\ar@{=}[d]\\
tM\ar[r]&E\ar[r]\ar[d]&N\ar[r]\ar[d]&tM[1]\\
&T'\ar@{=}[r]&T'
}\]
The triangle in the second column gives an exact sequence $0\to M\to E\to T'\to 0$ in $\A$. By the triangle in the third row,
we have $E\in\T(T)$ since $\T(T)$ is closed under extensions. Hence $\T(T)$ is a cogenerator of $\A$.
\end{proof}

\begin{remark}
By Proposition~\ref{prop:back}, 
an object $T$ in an $\Ext$-finite abelian category $\A$ is a tilting object if and only if 
$\Ext^i_\A(T,-)=0$ for $i>1$, $\Ext^1_\A(T,T)=0$ and condition (T5) holds. 
Note that in \cite{hrs}, the category $\A$ is only assumed to be $\Hom$-finite. 
However, if $\A$ has a tilting object $T$, 
then by \cite[Theorem 4.6]{hrs}, we have that $\DAA$ is equivalent to the bounded derived category of $\End_A(T)$. 
Since $\End_A(T)$ is a finite dimensional algebra, 
it follows that also $\DAA$ is Krull-Schmidt and $\Hom$-finite. 
Hence $\A$ is $\Ext$-finite. 

\end{remark}

Let $D^c(\A)$ be the full subcategory of $D^b(\A)$
consisting of the complexes $\x$ that satisfy
\[\Hom_{D^b(\A)}(\x,\A[i])=0\]
for $i>>0$.
It is clear that $D^c(\A)$ is a thick subcategory of $D^b(\A)$. 
Recall that an object $\p$ in a triangulated category $\T$ is called a silting object (see \cite[Definition 2.1]{ai}) if
\begin{itemize}
\item[-] $\Hom_\T(\p,\p[i])=0$ for any $i>0$, and
\item[-] $\thick \p=\T$,
\end{itemize}
where $\thick \p$ denotes the smallest thick subcategory of $\T$ containing $\p$.

\begin{lemma}\label{lem:thick}
Let $\p$ be a 2-term silting complex in $D^b(\A)$. Then $\p\in D^c(\A)$ and $\thick \p=D^c(\A)$. In particular, $\p$ is a silting object in $D^c(\A)$.
\end{lemma}

\begin{proof}

By (S1) we have that $\p$ belongs to $D^c(\A)$.
Let $\x$ be an object in $D^c(\A)$. In particular, by Proposition~\ref{prop:general} (b), the complex
$\x$ belongs to $D^{\leq s}(\p)$ for some integer $s$. 
Using \eqref{eq:resolution} we obtain
 \[D^{\leq s}(\p)=D^{\leq 0}(\p)[-s]=\left(\bigcup_{z>0}\A[z]\ast\cdots
 \ast\A[1]\ast\T(\p)\right)[-s]
\subset\left(\bigcup_{z>0}\A[z]\ast\cdots \ast\A[1]\ast\A\right)[-s].\]
So by definition, we obtain
\begin{equation}\label{eq:vanish}
\Hom_{D^b(\A)}(\x,D^{\leq s}(\p)[i])=0 \text{ for } i>>0.
\end{equation}
 Take a right $\add\p[-s]$-approximation $\p'[-s]\to\x$
and extend it to a triangle
\[\x_1\to\p'[-s]\to\x\to\x_1[1].\]
By applying $\Hom_{D^b(\A)}(\p,-)$ to this triangle,
we have that $\x_1$ is also in $D^{\leq s}(\p)$.
Then $\x\in\add\p[-s]\ast D^{\leq s}(\p)[1]$.
Recursively, we have that $\x$ is in
$\add\p[-s]\ast\add\p[-s+1]\ast\cdots\ast\add\p[-s+i-1]\ast D^{\leq s}(\p)[i]$
for any $i>0$.
Then there is a triangle
$$\x' \to \x \xrightarrow{u} \x'' \to \x'[1]$$
with $\x' \in \add\p[-s]\ast\add\p[-s+1]\ast\cdots\ast\add\p[-s+i-1]$
and $\x'' \in D^{\leq s}(\p)[i]$.
By \eqref{eq:vanish}, for $i >> 0$, we have that $u = 0$.
Hence $\x\in\thick\p$, which implies that $\thick \p=D^c(\A)$. By (S2) and Lemma~\ref{lem:pre}(b), we have $\Hom_\DAA(\p,\p[i])=0$ for any $i>0$. Hence it follows that $\p$ is a silting object in $D^c(\A)$.
\end{proof}

Consider now the case with $\A = \mod A$, for a finite
dimensional algebra $A$. Then we have two different
definitions of 2-term silting complexes, which we now compare.

\begin{corollary}\label{cor:alg}
Let $A$ be a finite dimensional $\k$-algebra. Regard $K^b(\proj A)$ as a thick subcategory of $D^b(A)$. Then $D^c(\mod A)=K^b(\proj A)$ and a complex in  $\DA$ satisfies Definition
\ref{def:silting} if and only if it is
a 2-term silting complex in $K^b(\proj A)$ as defined in the
introduction.
\end{corollary}

\begin{proof}
It is clear that $K^b(\proj A)\subset D^c(\mod A)$. Conversely,
it is straightforward to check that any complex $\x\in D^c(\mod A)$ has a projective resolution of finite length, so it is in $K^b(\proj A)$. Hence $D^c(\A)=K^b(\proj A)$.

Let $\p$ be a complex $\{d^i \colon P^i\to P^{i+1}\}_{i\in\mathbb{Z}}$ in $D^b(\mod A)$ satisfying (S1), (S2) and (S3). Then by Lemma~\ref{lem:thick}, the complex $\p$ is a silting object in $D^c(\A)=K^b(\proj A)$. Up to isomorphism, we may assume that $\p$
is minimal in the sense that $\im d^i \subseteq \rad P^{i+1}$
for all $i$.
By Lemma~\ref{lem:pre}(a), $H^i(\p)=0$ for $i>0$ or $i<-1$, so $P^i=0$ for $i\neq 0,-1$. Hence $\p$ is a 2-term silting complex in $K^b(\proj A)$.

Let $\p$ be a 2-term silting complex in $K^b(\proj A)$, as defined
in the introduction. Then it is clear that $\p$ satisfies (S1) and (S2). There is a triangle $A\to\p'\to\p''\to A[1]$, with
$\p', \p''$ in  $\add \p$, see \cite[Corollary~3.3]{bz}. Since $A$ satisfies (S3), so does $\p$. Thus, the proof is finished.
\end{proof}

We also have the following application of Lemma~\ref{lem:thick}.

\begin{corollary}\label{cor:finite}
Let $\A$ be an $\Ext$-finite abelian category satisfying that for any $M\in\A$, there is an $m\in\mathbb{Z}$ such that $\Ext_\A^i(M,-)=0$ for any $i>m$.
Then $D^c(\A)=D^b(\A)$ and the 2-term silting complexes in $D^b(\A)$
are precisely the silting objects in $D^b(\A)$ satisfying (S1).
\end{corollary}

\begin{proof}
For any object $M\in\A$, we have $M\in D^c(\A)$
by assumption.
Since the smallest thick subcategory of $\DAA$ containing $\A$ is $\DAA$,
we have $D^c(\A)=\DAA$. Then the last part of the assertion follows directly from Lemma~\ref{lem:thick}.
\end{proof}

\begin{definition}
Let $B$ be a finite dimensional $\k$-algebra. We call $B$ a quasi-silted algebra if there is an $\Ext$-finite hereditary abelian $\k$-category $\H$ and a 2-term silting complex $\p$ in $D^b(\H)$ such that $B=\End_{D^b(\H)}(\p)$.
\end{definition}

\begin{lemma}\label{lem:dec}
Let $\p$ be a two-term silting complex in the bounded derived category of an $\Ext$-finite hereditary abelian category $\H$. Then $\p\cong H^0(\p)\oplus H^{-1}(\p)[1]$ and $H^{-1}(\p)$ is projective in $\H$.
\end{lemma}

\begin{proof}
Since $\H$ is hereditary, triangle \eqref{eq:t1} is split. Hence we have $\p\cong H^0(\p)\oplus H^{-1}(\p)[1]$. Then by (S1), it follows that $H^{-1}(\p)$ is projective in $\H$.
\end{proof}

\begin{corollary}\label{cor:tilting}
Let $\H$ be an $\Ext$-finite hereditary abelian category such that there are 2-term silting complexes in $D^b(\H)$. Then $\H$ has tilting objects.
\end{corollary}

\begin{proof}
Let $\p$ be a 2-term silting complex in $D^b(\H)$.
By Lemma~\ref{lem:dec}, we have $\p\cong H^0(\p)\oplus H^{-1}(\p)[1]$. 
Consider a triangle containing a right $\add H^0(\p)$-approximation of $H^{-1}(\p)[1]$
\[\x'\to \x\to H^{-1}(\p)[1]\to\x'[1],\]
where $\x\in\add H^0(\p)$. Since $\p$ is a silting object in $D^b(\H)$ by Corollary~\ref{cor:finite}, the complex $\x'\oplus H^0(\p)$ is also a silting object in $D^b(\H)$ by \cite[Theorem~2.31]{ai}. It is easy to check that $\x'\oplus H^0(\p)$ satisfies condition (S1),
so by Corollary~\ref{cor:finite} again,
$\x'\oplus H^0(\p)$ is a 2-term silting complex in $D^b(\H)$.
Furthermore, we have that $\x'\in H^{-1}(\p)\ast \x$ implies  $\x'\in\H$. Hence $\x'\oplus H^0(\p)$ is in $\H$. By Proposition~\ref{prop:back}, it is a tilting object.
\end{proof}

Now we have the following direct consequence,
which, together with Theorem~\ref{thm:1-silt}, Proposition~\ref{prop:back} and \cite[Theorem~II.2.3]{hrs}, also finishes the proof of part (b) of our main result, Theorem \ref{main}.

\begin{corollary}
Any quasi-silted algebra is shod.
\end{corollary}

\begin{proof}
Let $\p$ be a 2-term silting complex in $D^b(\H)$ for an $\Ext$-finite hereditary abelian $\k$-category $\H$. By Corollary~\ref{cor:tilting}, it follows that $\H$ has tilting objects. Without loss of generality, we may assume that $\H$ is indecomposable. Then either $\H$ has enough projective objects or $\H$ does not have any projective objects, by \cite[Theorem 4.2]{h2}. For the first case, we have that $\H\simeq\mod H$ for some finite-dimensional hereditary $\k$-algebra $H$ and then $\End_{D^b(\H)}(\p)$ is shod by Theorem~\ref{thm:1-silt}. For the second case, by Lemma~\ref{lem:dec}, we have $\p\cong H^0(\p)$ and hence $\p$ is isomorphic to a tilting object by Proposition~\ref{prop:back}. Hence $\End_{D^b(\H)}(\p)$ is quasi-tilted and hence shod by \cite[Theorem II.2.3]{hrs}.
\end{proof}

\section{An example}\label{sec:example}

In this section we discuss a small example of
a strictly shod algebra, and point out how it can be
realized as the endomorphism algebra of a
2-term silting complex over a hereditary algebra.

Consider the algebra $B= \k Q/J$, where $Q$
is the Dynkin quiver of type $A_4$, with linear orientation
$$
\xymatrix{
1 & 2 \ar^{\alpha}[l] & 3 \ar^{\beta}[l] & 4 \ar^{\gamma}[l]
}
$$
and with ideal of relations  $J$ generated by $\beta\alpha$ and $\gamma\beta$. The global dimension of $B$ is 3.

This is a Nakayama algebra, it has exactly 7
(isomorphism classes of) indecomposable  modules.
Out of these, 5 are projective and/or injective.
In addition, we have the simples $S_2$ and $S_3$,
corresponding to vertex 2 and 3. It is easily verified that $S_2$ (resp. $S_3$) has projective dimension 1 (resp. 2), and injective dimension 2 (resp. 1).
So this is by definition a strictly shod algebra. It is easily
seen to be derived equivalent to a path algebra of
type $A_4$ (it can be obtained from $A_4$ by tilting twice).

Now consider the hereditary path algebra
$H = \k \overset{\rightarrow}{D_4}$ with $\overset{\rightarrow}{D_4}$ the quiver

$$
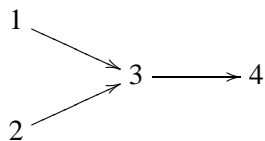
\xymatrix@C=33pt@R=7pt{
1 \ar[dr] & & \\
 & 3 \ar[r] & 4  \\
 2 \ar[ur] & &
}
$$

Let $P_i$ denote the projective $H$-module corresponding to vertex $i$. Consider the 2-term silting complex given
by $\p = \p_L \oplus \p_M \oplus \p_R$, with
$\p_L  = P_2[1]$, with $\p_M  = (P_{3} \to P_1)$ and with
$\p_R  = P_1 \oplus P_4$.
Then, it is easy to verify that $\End_{D^b(\mod H)}(\p) \cong B$.

We remark that by \cite[Section 3]{air}, there is a 1--1 correspondence between 2-term silting complexes and
so called support $\tau$-tilting modules for a given
algebra, given by $\p \mapsto H^0(\p)$.
Note that when the algebra is hereditary, support
$\tau$-tilting is the same as support tilting.
The support tilting module corresponding to $\p$
in our example is given by $P_1 \oplus P_4 \oplus
P_1/P_3$, which is easily seen to be a tilting module
for the path algebra of the subquiver spanned by the vertices
${1,3,4}$.

\end{document}